\documentclass{article}
\usepackage [utf8] {inputenc}
\usepackage{url}
\usepackage{mathtools}
\usepackage{amsthm}
\usepackage{amssymb}

\newtheorem{theorem}{Theorem}[section]
\newtheorem{proposition}[theorem]{Proposition}
\newtheorem{lemma}[theorem]{Lemma}
\newtheorem{corollary}[theorem]{Corollary}

\DeclareMathOperator{\X}{\mathfrak{X}}
\DeclareMathOperator{\diam}{diam}
\newcommand{\diamd}{\overrightarrow{\mathrm{diam}}}

\title{On directed and undirected diameters of vertex-transitive graphs}
\author{Saveliy V. Skresanov}
\date{}

\begin{document}
\maketitle

\begin{abstract}
	A directed diameter of a directed graph is the maximum possible distance between
	a pair of vertices, where paths must respect edge orientations, while undirected diameter
	is the diameter of the undirected graph obtained by symmetrizing the edges. In 2006 Babai proved that for a connected directed Cayley graph
	on \( n \) vertices the directed diameter is bounded above by a polynomial in undirected diameter and \( \log n \).
	Moreover, Babai conjectured that a similar bound holds for vertex-transitive graphs. We prove this conjecture of Babai,
	in fact, it follows from a more general bound for connected relations of homogeneous coherent configurations.
	The main novelty of the proof is a generalization of Ruzsa's triangle inequality from additive combinatorics to the setting of graphs.
\end{abstract}

\section{Introduction}

Given a (finite) connected directed graph \( \Gamma \), let \( \diamd(\Gamma) \) denote the directed diameter,
which is the maximum directed distance between a pair of vertices of \( \Gamma \). The undirected diameter \( \diam(\Gamma) \)
is the maximum undirected distance between a pair of vertices, that is, we forget about edge orientations in \( \Gamma \)
when measuring distances. Clearly \( \diam(\Gamma) \leq \diamd(\Gamma) \).

In general, it is not possible to bound the directed diameter in terms of the undirected diameter only,
yet it is interesting how close these two quantities are if the graph is symmetric enough. One of the classical % TODO: symmetric is a bad word
examples when the difference between the directed and undirected diameters manifests itself, is the famous
question of Babai about the diameters of Cayley graphs of finite simple groups. Given a finite group \( G \) and its generating set \( S \subseteq G \),
recall that a Cayley graph \( \mathrm{Cay}(G, S) \) is a directed graph with vertex set \( G \), where two vertices \( x, y \in G \) are joined
by a directed edge if \( y = xg \) for some \( g \in S \). Let \( \log n \) denote the base-2 logarithm of \( n \).
Babai proposed the following conjecture:%
\medskip

\noindent
\cite[Conjecture~1.7]{babaiSeress} \emph{Given a nonabelian finite simple group \( G \) generated by \( S \subseteq G \), we have
\[ \diam(\mathrm{Cay}(G, S)) \leq (\log |G|)^C \]
for some universal constant \( C > 0 \).}
\medskip

Although this conjecture is still open, there have been a plethora of results solving the conjecture in some particular cases
and obtaining nontrivial bounds for the diameter (see~\cite{helfgott, pyber, BGT, babaiSeress, BBS}, to name a few).
These works often make heavy use of the fact that we may assume our generating set \( S \) to be closed under taking inverses,
for example, the proof of~\cite{BBS} requires taking commutators of group elements, i.e.\ words of the form \( x^{-1}y^{-1}xy \) for \( x, y \in G \).
Such a word does not correspond to a proper path in the directed Cayley graph, so one has to work with undirected graphs instead.

Luckily, in 2006 Babai proved a result which resolved this issue for Cayley graphs. % TODO: define O(.)?

\begin{proposition}[{\cite[Theorem~1.4]{euler}}]\label{babCay}
	Let \( \Gamma \) be a connected Cayley graph on \( n \) vertices. Then \( \diamd(\Gamma) = O(\diam(\Gamma)^2 (\log n)^3) \).
\end{proposition}

In particular, it immediately follows that if \( \diam(\mathrm{Cay}(G, S)) \leq (\log |G|)^C \) for some group \( G \) and a generating set \( S \),
then \( \diamd(\mathrm{Cay}(G, S)) = O((\log |G|)^{2C+3}) \), so it is enough to solve Babai's conjecture for undirected Cayley graphs only.

The proof of Proposition~\ref{babCay} relies on a result of Babai and Erd\H{o}s~\cite{erdos} about fast generating sets of finite groups,
and a theorem of Babai and Szegedy on expansion in vertex-transitive graphs~\cite{szegedy}. Since Cayley graphs are vertex-transitive, it is natural
to ask if a similar bound on the directed diameter holds for vertex-transitive graphs. To this end, Babai makes the following conjecture:%
\medskip

\noindent
\cite[Conjecture~6.1]{euler} \emph{Given a connected vertex-transitive graph \( \Gamma \) on \( n \) vertices, there exists a polynomial bound on \( \diamd(\Gamma) \) in terms of \( \diam(\Gamma) \) and \( \log n \).}
\medskip

In the same paper Babai found such a bound under a condition that the outdegree of \( \Gamma \) is bounded.
If \( k \) is the outdegree of \( \Gamma \), then in~\cite[Theorem~1.1]{euler} it is proved that \( \diamd(\Gamma) = O(\diam(\Gamma) \cdot k \log n) \)
and in~\cite[Corollary~1.2]{euler} that \( \diamd(\Gamma) = O(\diam(\Gamma)^2 \cdot k \log k) \).
If \( \Gamma \) is edge-transitive, then by~\cite[Theorem~1.5]{euler} we have \( \diamd(\Gamma) = O(\diam(\Gamma) \log n) \), so the conjecture holds in this case.

The main result of this paper is the solution of the above conjecture of Babai on vertex-transitive graphs.

\begin{theorem}\label{main}
Let \( \Gamma \) be a connected vertex-transitive graph on \( n \) vertices. Then
\[ \diamd(\Gamma) = O(\diam(\Gamma)^2 (\log n)^2). \]
\end{theorem}

Note that in the case of Cayley graphs our result improves Babai's by a factor of \( \log n \).

We derive Theorem~\ref{main} as a corollary of a more general result about homogeneous coherent configurations (also known as association schemes~\cite{BI}).
Namely, in Theorem~\ref{mainS} we prove that the above bound for directed diameter holds when \( \Gamma \) is a connected relation of a homogeneous coherent configuration.
Since a vertex-transitive graph is a relation of a coherent configuration corresponding to its full automorphism group, Theorem~\ref{main} follows at once.

To prove Theorem~\ref{mainS} we employ a vertex expansion bound for undirected relations of homogeneous coherent configurations (Proposition~\ref{expand}),
analogous to the Babai-Szegedy bound~\cite{szegedy} for undirected graphs.
In order to relate the vertex expansion of the original directed graph with expansion of its symmetrization, we introduce the \emph{Ruzsa triangle inequality
for graphs} (Theorem~\ref{ruzsa}), which we think might be of independent interest. To state the result, let \( a, b \subseteq \Omega \times \Omega \) be some relations (or, equivalently,
directed graphs) on the set \( \Omega \) and let \( ab \subseteq \Omega \times \Omega \) denote their product. Let \( b^* \subseteq \Omega \times \Omega \) denote the transposed relation,
and let \( \| a \| \) be the maximum outdegree of a vertex in \( a \) (we give precise definitions of these notions in Sections~\ref{pre} and~\ref{sr}).
If \( a, b, c \subseteq \Omega \times \Omega \) and \( b \) is regular, then
\[ \| ac \| \cdot \| b \| \leq \| ab \| \cdot \| b^*c \|. \]
When \( a, b, c \) are Cayley graphs, this gives the classical Ruzsa inequality for groups~\cite{ruzsaTr}, an indispensable tool in additive combinatorics.
In Section~\ref{sr} we show how this can be used to bound the directed diameter when the underlying coherent configuration
has a certain commutativity condition (Proposition~\ref{commBound}). For instance, if \( \Gamma \) is a connected Cayley graph on \( n \) vertices over an abelian group, then
\[ \diamd(\Gamma) = O(\diam(\Gamma) \log \log n). \]

In Section~\ref{neg} we show that in general it is not possible to bound the directed diameter or directed girth in terms of
the undirected diameter only, even in the case of Cayley graphs over abelian groups. This answers two other questions of Babai~\cite[Problems~6.4 and~6.5]{euler},
and the proof relies on a construction of Haight and Ruzsa~\cite{haight, ruzsa} of special subsets of cyclic groups.

The structure of the paper is as follows. In Section~\ref{pre} we give preliminaries on coherent configurations and
prove a generalization of the Babai-Szegedy bound~\cite{szegedy} to relations of homogeneous coherent configurations (Proposition~\ref{expand}).
In Section~\ref{sr} we prove the Ruzsa inequality for graphs (Theorem~\ref{ruzsa})
and give an application for coherent configurations with a commutativity condition (Proposition~\ref{commBound}).
The main result is proved in Section~\ref{prMain}, and in Section~\ref{neg} we give examples of Cayley graphs with bounded undirected diameter
and unbounded directed diameter and girth.

\section{Preliminaries}\label{pre}

Let \( \Omega \) be a finite set. Given two relations \( a, b \subseteq \Omega \times \Omega \) let
\[ ab = \{ (\alpha, \beta) \in \Omega \times \Omega \mid (\alpha, \gamma) \in a,\, (\gamma, \beta) \in b \text{ for some } \gamma \in \Omega \} \]
denote the product relation. We write \( 1_\Omega = \{ (\alpha, \alpha) \mid \alpha \in \Omega \} \) for the diagonal relation,
and \( a^* = \{ (\beta, \alpha) \in \Omega \times \Omega \mid (\alpha, \beta) \in a \} \) for the transposition of~\( a \).
Note that \( (a^*)^* = a \) and \( (ab)^* = b^* a^* \).

We will view \emph{graphs} on the vertex set \( \Omega \) as relations on \( \Omega \), and will use these two terms interchangeably.
All our graphs are directed by default, and we say that a graph \( a \subseteq \Omega \times \Omega \) is \emph{undirected}, if \( a = a^* \).
If \( a \) is a connected graph on \( \Omega \), we write \( \diamd(a) \) for the \emph{directed diameter} of \( a \), that is,
the largest possible distance between a pair of vertices in \( a \), where paths between vertices must preserve edge orientations.
Similarly, \( \diam(a) \) denotes the \emph{undirected diameter} of \( a \), which is the largest distance between a pair of vertices
where paths are not required to preserve edge orientations; in other words, \( \diam(a) = \diamd(a \cup a^*) \).

Let \( S \) be a partition of \( \Omega \times \Omega \) into relations. We say that a tuple \( \X = (\Omega, S) \)
is a homogeneous coherent configuration or, shortly, a \emph{scheme}, if the following holds~\cite[Definition~2.1.3]{ponom}:
\begin{enumerate}
	\item \( 1_\Omega \in S \),
	\item If \( a \in S \), then \( a^* \in S \),
	\item For any \( a, b, c \in S \) and any \( (\alpha, \beta) \in a \), the number
		\[ |\{ \gamma \in \Omega \mid (\alpha, \gamma) \in b,\, (\gamma, \beta) \in c \}| \]
		does not depend on the choice of \( (\alpha, \beta) \in a \).
\end{enumerate}

The set \( S = S(\X) \) is the set of \emph{basis relations} of \( \X \), and we define \( S^\cup = S^\cup(\X) \) as the set of all unions
of basis relations of \( \X \), in particular, \( S \subseteq S^\cup \). It follows from the definition of a scheme that \( S^\cup \)
is closed under taking products and transpositions of relations~\cite[Proposition~2.1.4]{ponom}. Moreover, if \( a \in S^\cup \) then the indegree and outdegree of
any vertex of \( a \) are the same and do not depend on the vertex~\cite[Definition~2.1.10 and Corollary~2.1.13]{ponom}. In particular, all graphs (relations) from \( S^\cup \) are regular.

Given a transitive permutation group \( G \) on \( \Omega \) one can construct a scheme as follows (see also~\cite[Definition~2.2.3]{ponom}). We set \( \X = (\Omega, S) \)
where \( S \) is the partition of the set \( \Omega \times \Omega \) into orbits of \( G \) in its natural action on pairs; one can easily check that \( \X \) is indeed a scheme.
If \( \Gamma \) is a vertex-transitive graph (for example, a Cayley graph) on \( \Omega \) with the full automorphism group \( G \), then \( \Gamma \) can be partitioned
into orbits of \( G \) and hence \( \Gamma \in S^\cup(\X) \).
\medskip

The following lemma is a generalization of the third property from the definition of a scheme.

\begin{lemma}[{\cite[Exercise~2.7.25]{ponom}}]\label{paths}
	Let \( \X = (\Omega, S) \) be a scheme, and let \( r \), \( r_1, \dots, r_{m-1} \in S \), \( m \geq 2 \)
	and \( (u, w) \in r \). Then the number of tuples \( (v_1, \dots, v_m) \in \Omega^m \) such that \( (v_1, v_m) = (u, w) \) and \( (v_i, v_{i+1}) \in r_i \),
	\( i = 1, \dots, m-1 \) does not depend on the choice of \( (u, w) \in r \).
\end{lemma}

Let \( b \subseteq \Omega \times \Omega \) be an undirected graph on \( \Omega \).
For a subset \( T \subseteq \Omega \) let \( \partial_b(T) \) denote the set of vertices outside of \( T \) adjacent in \( b \) to a vertex in \( T \), i.e.\
\[ \partial_b(T) = \{ \beta \in \Omega \setminus T \mid \alpha \in T,\, (\alpha, \beta) \in b \}. \]
Babai and Szegedy~\cite[Theorem~2.2]{szegedy} showed that for a connected vertex-transitive graph \( b \) and \( T \subseteq \Omega \), \( 0 < |T| \leq |\Omega|/2 \),
the vertex expansion ratio \( |\partial_b(T)|/|T| \) can be bounded from below in terms of \( \diam(b) \).
To prove Theorem~\ref{mainS}, we show that this bound holds in the setting of schemes.
The argument is essentially the same as in~\cite{szegedy}, but requires a few technicalities to count paths between vertices in relations of schemes.

\begin{proposition}\label{expand}
	Let \( \X = (\Omega, S) \) be a scheme, and let \( b \in S^\cup \) be a connected relation with \( b = b^* \).
	For any nonempty subset \( T \) of \( \Omega \) we have
	\[ \frac{|\partial_b(T)|}{|T|} \geq \frac{2(1 - |T|/|\Omega|)}{\diam(b) + |T|/|\Omega|}. \]
	In particular, if \( |T| \leq |\Omega|/2 \), then
	\[ \frac{|\partial_b(T)|}{|T|} \geq \frac{2}{2\diam(b) + 1}. \]
\end{proposition}
\begin{proof}
	We may view \( b \) as an undirected graph on \( \Omega \). For \( x, y \in \Omega \)
	let \( d(x, y) \) denote the distance between \( x \) and \( y \) in \( b \).
	A path \( x_0, x_1, \dots, x_m \) in \( b \) will be called a geodesic if \( m = d(x_0, x_m) \).
	Note that a geodesic is an ordered sequence of points, so \( x_m, \dots, x_1, x_0 \) is a different geodesic.

	By~\cite[Theorem~2.6.7]{ponom} the distance-\( i \) graph \( \{ (u, v) \in \Omega \times \Omega \mid d(u, v) = i \} \)
	lies in \( S^{\cup} \), in particular, the distance between points \( x \) and \( y \) depends
	only on the relation where \( (x, y) \) lies.

	Let \( p(x, y) \) denote the number of geodesics between \( x \) and \( y \). If \( z \) is some vertex, then
	we claim that the number
	\[ P_z = \sum_{x, y \in V} \frac{1}{p(x, y)} |\{ L \mid L \text{ is a geodesic from } x \text{ to } y \text{ passing through } z \}| \]
	does not depend on the choice of \( z \in \Omega \). Indeed, by Lemma~\ref{paths}, the number \( p(x, y) \) depends only on the relation
	where \( (x, y) \) lies. The number of geodesics from \( x \) to \( y \) passing through \( z \) can be expressed as
	\[
		N_z(x, y) = 
		\begin{cases}
			p(x, z) \cdot p(z, y), & \text{ if } d(x, z) + d(z, y) = d(x, y),\\
			0, & \text{ otherwise.}
		\end{cases}
	\]
	By the paragraph above, this number depends only on the relations where \( (x, z) \), \( (z, y) \) and \( (x, y) \) lie.
	Therefore we have the following formula:
	\[ P_z = \sum_{r, s, t \in S} \sum_{\substack{x, y \in \Omega,\\(x, z) \in r, (z, y) \in s, (x, y) \in t}} \frac{1}{p(x, y)} N_z(x, y). \]
	By applying Lemma~\ref{paths} to the path \( z, x, y, z \), we see that the number of tuples \( (x, y) \in t \) satisfying
	\( (x, z) \in r \), \( (z, y) \in s \) does not depend on the choice of \( z \). Hence \( P = P_z \) does not depend on the choice of \( z \), as claimed.
	
	Set \( n = |\Omega| \). We claim that
	\[ n \cdot P = \sum_{x, y \in \Omega} (d(x, y)+1). \]
	Since \( P \) does not depend on the choice of \( z \in \Omega \) we have
	\begin{multline*}
		n \cdot P = \sum_{z \in \Omega} \sum_{x, y \in \Omega} \frac{1}{p(x, y)} |\{ L \mid L \text{ is a geodesic from } x \text{ to } y \text{ passing through } z \}|\\
		= \sum_{x, y \in \Omega} \frac{1}{p(x, y)} \sum_{z \in \Omega} |\{ L \mid L \text{ is a geodesic from } x \text{ to } y \text{ passing through } z \}|.
	\end{multline*}
	There are \( d(x,y)+1 \) vertices in a geodesic from \( x \) to \( y \), and there are \( p(x, y) \) such geodesics, hence
	\begin{multline*}
		\sum_{x, y \in \Omega} \frac{1}{p(x, y)} \sum_{z \in \Omega} |\{ L \mid L \text{ is a geodesic from } x \text{ to } y \text{ passing through } z \}|\\
		= \sum_{x, y \in \Omega} \frac{1}{p(x, y)} \cdot (d(x, y)+1)p(x, y) = \sum_{x, y \in \Omega} (d(x,y)+1),
	\end{multline*}
	which proves the claim.

	Now we are ready to prove the expansion bound. Let \( T \) be a nonempty subset of \( \Omega \).
	Consider a set of pairs
	\[ K = (T \times (\Omega \setminus T)) \cup ((\Omega \setminus T) \times (T \cup \partial_b(T))). \]
	A path between \( T \) and \( \Omega \setminus T \) must intersect \( \partial_b(T) \), hence
	\[ |K| =
	\sum_{(x, y) \in K} \frac{1}{p(x, y)} |\{ L \mid L \text{ is a geodesic from } x \text{ to } y \text{ through } \partial_b(T)\}|. \]
	We extend summation to the whole of \( \Omega \) and derive
	\begin{multline*}
		|K| \leq
		\sum_{x, y \in \Omega} \frac{1}{p(x, y)} |\{ L \mid L \text{ is a geodesic from } x \text{ to } y \text{ through } \partial_b(T)\}|\\
		\leq \sum_{x, y \in \Omega} \frac{1}{p(x, y)} \sum_{z \in \partial_b(T)} |\{ L \mid L \text{ is a geodesic from } x \text{ to } y \text{ through } z\}| \\
		= \sum_{z \in \partial_b(T)} \sum_{x, y \in \Omega} \frac{1}{p(x, y)} |\{ L \mid L \text{ is a geodesic from } x \text{ to } y \text{ through } z\}| \\
		= \sum_{z \in \partial_b(T)} P = |\partial_b(T)| \cdot P = \frac{|\partial_b(T)|}{n} \sum_{x, y \in \Omega} (d(x, y)+1) \leq n \cdot |\partial_b(T)| \cdot (\diam(b)+1).
	\end{multline*}
	To compute the size of \( K \), recall that \( T \) and \( \partial_b(T) \) are disjoint, hence
	\[ |K| = |T \times (\Omega \setminus T)| + |(\Omega \setminus T) \times (T \cup \partial_b(T))| = (n-|T|)(2|T| + |\partial_b(T)|). \]
	After substituting in the inequality above and dividing by \( n|T| \) we derive
	\[ \left(1 - \frac{|T|}{n}\right)\left(2 + \frac{|\partial_b(T)|}{|T|}\right) \leq \frac{\partial_b(T)}{|T|}(\diam(b)+1). \]
	The claimed result follows after rearranging terms in the inequality.
\end{proof}

We note that a similar bound for the edge expansion of basis relations of coherent configurations was proved in~\cite[Proposition~2.8]{pybSkr}.

\section{Ruzsa inequality for graphs}\label{sr}

One of the central tools in additive combinatorics is the Ruzsa triangle inequality~\cite{ruzsaTr}.
Let \( G \) be an arbitrary group. Given subsets \( A, B \subseteq G \) let \( AB = \{ ab \mid a \in A, b \in B \} \)
denote the product set and let \( A^{-1} = \{ a^{-1} \mid a \in A \} \) be the inverse set.
If \( A, B, C \subseteq G \) are finite subsets, then the Ruzsa triangle inequality claims that
\[ |AC| \cdot |B| \leq |AB| \cdot |B^{-1}C|. \]
The naming stems from the fact that Ruzsa's inequality implies that the function
\[ d(A, B) = \log\frac{|AB^{-1}|}{(|A|\cdot|B|)^{1/2}} \]
obeys the usual triangle inequality \( d(A, C) \leq d(A, B) + d(B, C) \).

The proof of this beautiful result is rather simple: we construct an injective mapping \( f : AC \times B \to AB \times B^{-1}C \)
in the following manner. For all \( x \in AC \) fix some decomposition \( x = a_xc_x \), \( a_x \in A \), \( c_x \in C \).
Now for \( b \in B \) set \( f(x, b) = (a_xb,\, b^{-1}c_x) \). This map is indeed injective, since if \( f(x, b) = (y, z) \)
then \( x = a_xc_x = (a_xb)(b^{-1}c_x) = yz \), and \( b = (a_x)^{-1}y \); in other words, one can recover \( x \) and \( b \)
from \( y \) and \( z \).

We will now show that essentially the same argument applies to finite graphs. Let \( \Omega \) be some finite set.
Given a relation \( a \subseteq \Omega \times \Omega \) and \( \omega \in \Omega \),
let \( \omega a = \{ \beta \in \Omega \mid (\omega, \beta) \in a \} \) denote the neighborhood of \( \omega \) in \( a \).
For a subset \( T \subseteq \Omega \) we similarly define \( Ta = \{ \beta \in \Omega \mid (\alpha, \beta) \in a,\, \alpha \in T \} \).

To make our results visually closer to analogous statements from additive combinatorics,
we will write \( \| a \| \) for the maximum outdegree of a relation \( a \subseteq \Omega \times \Omega \), i.e.\ \( \| a \| = \max_{\omega \in \Omega} |\omega a| \).
For any relations \( a, b \subseteq \Omega \times \Omega \) we trivially have \( \| ab \| \leq \| a \| \cdot \| b \| \).
We say that the relation \( a \) is \emph{regular}, if \( |\omega a| \) does not depend on the choice of \( \omega \in \Omega \).

\begin{theorem}[Ruzsa triangle inequality for graphs]\label{ruzsa}
	Let \( \Omega \) be a finite set, and let \( a, b, c \subseteq \Omega \times \Omega \).
	If \( b \) is regular, then
	\[ \| ac \| \cdot \| b \| \leq \| ab \| \cdot \| b^*c \|. \]
\end{theorem}
\begin{proof}
	Fix a point \( \omega \in \Omega \) such that \( \| ac \| = |\omega ac| \).
	Set \( m = \| ac \| \) and \( k = \| b \| \). Then \( m = |\omega ac| \), so \( \{ \gamma_1, \dots, \gamma_m \} = \omega ac \)
	for some necessarily distinct points \( \gamma_i \). Choose points \( \alpha_i \), \( i = 1, \dots, m \),
	such that \( (\omega, \alpha_i) \in a \) and \( (\alpha_i, \gamma_i) \in c \).
	For each \( i \in \{ 1, \dots, m \} \) choose \( k \) distinct points \( \beta_{i1}, \dots, \beta_{ik} \)
	such that \( (\alpha_i, \beta_{ij}) \in b \) for all \( i = 1, \dots, m \), \( j = 1, \dots, k \).

	Define a map \( f : \{ 1, \dots, m \} \times \{ 1, \dots, k \} \to \Omega \times \Omega \)
	by the rule \( f(i, j) = (\beta_{ij}, \gamma_i) \).
	Let \( I = \{ f(i, j) \mid i = 1, \dots, m,\, j = 1, \dots, k \} \) be the image of that map.
	Then
	\[ I = \bigcup_{\beta \in \omega ab} \{ f(i, j) \mid \beta_{ij} = \beta,\, i = 1, \dots, m,\, j = 1, \dots, k \}. \]
	Observe that sets in that union are disjoint, therefore
	\[ |I| = \sum_{\beta \in \omega ab}  |\{ \gamma_i \mid \beta_{ij} = \beta,\, i = 1, \dots, m,\, j = 1, \dots, k \}|. \]
	Since \( (\beta_{ij}, \alpha_i) \in b^* \) and \( (\alpha_i, \gamma_i) \in c \), we have \( \gamma_i \in \beta_{ij}b^*c \).
	Hence
	\[ |I| \leq \sum_{\beta \in \omega ab}  |\{ \gamma_i \in \beta b^*c \mid i = 1, \dots, m \}| \leq \sum_{\beta \in \omega ab} \| b^*c \|
	   \leq \| ab \| \cdot \| b^*c \|. \]

	Now we show that \( f \) is injective. Suppose \( f(i, j) = f(i', j') \) for some \( i \in \{ 1, \dots, m \} \)
	and \( j \in \{ 1, \dots, k \} \). Then \( \beta_{ij} = \beta_{i'j'} \) and \( \gamma_i = \gamma_{i'} \).
	Last equality implies \( i = i' \), so \( \beta_{ij} = \beta_{ij'} \) and \( j = j' \) follows.

	Since \( f \) is injective, we have \( m \cdot k \leq |I| \) which is the desired inequality.
\end{proof}

As all relations of a scheme are regular, we have the following immediate corollary.

\begin{corollary}[Ruzsa triangle inequality for schemes]
	Let \( \X \) be a scheme and \( a, b, c \in S^\cup(\X) \). Then
	\[ \| ac \| \cdot \| b \| \leq \| ab \| \cdot \| b^*c \|. \]
\end{corollary}

To illustrate how this can be used, we apply Ruzsa's inequality to bound the directed diameter
in schemes with a certain commutativity condition.

\begin{lemma}\label{commInd}
	For a scheme \( \X \) and \( a, b \in S^\cup(\X) \) we have
	\[ \| aa^* \|^{1/2} \cdot \| b \|^{1/2} \leq \| ab \|. \]
\end{lemma}
\begin{proof}
	Ruzsa's inequality gives
	\( \| aa^* \| \cdot \| b \| \leq \| ab \| \cdot \| b^*a^* \| = \| ab \|^2 \),
	where the last equality uses the fact that \( \| c \| = \| c^* \| \) for any \( c \in S^\cup(\X) \).
\end{proof}

Note that if \( a \) is a connected relation from \( S^\cup(\X) \) for some scheme \( \X \) on \( \Omega \), then \( \diamd(a) \) is equal to the smallest \( k \) such that
\[ (a \cup 1_\Omega)^k = \underbrace{(a \cup 1_\Omega) \cdot \dots \cdot (a \cup 1_\Omega)}_{k \text{ times}} = \Omega \times \Omega. \]
Here we used that \( 1_\Omega \) acts as a multiplicative identity, i.e.\ \( a 1_\Omega = 1_\Omega a = a \).
In its turn, \( \diam(a) \) is equal to the smallest \( k \) such that \( (a \cup a^* \cup 1_\Omega)^k = \Omega \times \Omega \).

Observe also that if \( a \in S^\cup(\X) \) is such that \( \| a \| > n/2 \) where \( n = |\Omega| \), then \( \| a^* \| > n/2 \),
and \( aa = \Omega \times \Omega \) by pigeonhole principle.

\begin{proposition}\label{commBound}
	Let \( \X \) be a scheme on \( n \) points and let \( a \in S^\cup(\X) \) be a connected relation such that \( aa^* = a^*a \).
	Then \( \diamd(a) = O(\diam(a) \log \log n) \).
\end{proposition}
\begin{proof}
	Set \( k = \diam(a) \). We may always assume that \( 1_\Omega \subseteq a \), so \( (a \cup a^*)^k = \Omega \times \Omega \).
	By expanding the brackets and using commutativity we see that \( a^k (a^*)^k = \Omega \times \Omega \), in particular,
	\( \| a^k (a^*)^k \| = n \). Since \( (a^*)^k = (a^k)^* \) we have \( \| a^k (a^k)^* \| = n \) and thus \( \| a^k \| \geq n^{1/2} \).

	By Lemma~\ref{commInd}, for any \( m \geq 2 \) we have
	\[ \| a^{m \cdot k} \| \geq \| a^k (a^k)^* \|^{1/2} \cdot \| a^{(m-1)k} \|^{1/2} \geq n^{1/2} \cdot  \| a^{(m-1)k} \|^{1/2}, \]
	so by induction we derive
	\[ \| a^{m\cdot k} \| \geq n^{\frac{1}{2} + \frac{1}{4} + \dots + \frac{1}{2^m}} = n^{1 - \frac{1}{2^m}}, \]
	for any positive integer \( m \). Now set \( m = \lceil \log \log n \rceil + 1 \). Then \( \| a^{m \cdot k} \| > n/2 \)
	and hence \( \| (a^{m \cdot k})^* \| > n/2 \), so \( a^{2m \cdot k} = \Omega \times \Omega \). Therefore
	\[ \diamd(a) \leq 2\diam(a) (\lceil \log \log n \rceil + 1), \]
	as claimed.
\end{proof}

A Cayley graph can be viewed as a relation of a certain \emph{Cayley scheme}~\cite[Definition~2.4.1]{ponom}, and the condition
\( aa^* = a^*a \) of the proposition above simply means that the generating set of the Cayley graph must commute with its inverse.
Hence the above proposition applies to Cayley graphs over abelian groups or, more generally, to Cayley graphs with normal connection sets, so we have
a considerable improvement to Babai's bound in these cases. Connected relations of \emph{commutative schemes}~\cite[Section~2.3.1]{ponom} also satisfy the conditions of the proposition.

\section{Proof of the main result}\label{prMain}

Theorem~\ref{main} is a consequence of the following more general result about relations of schemes.

\begin{theorem}\label{mainS}
	Let \( \X = (\Omega, S) \) be a scheme on \( n = |\Omega| \) points, and let \( a \in S^\cup \) be a connected relation. Then
	\[ \diamd(a) = O(\diam(a)^2 (\log n)^2). \]
\end{theorem}
\begin{proof}
	Without loss of generality we may assume that \( 1_\Omega \subseteq a \).
	Let \( b = a \cup a^* \) be the symmetrization of \( a \).

	Suppose \( t \in S^\cup \) is such that \( \| t \| \leq n/2 \). Choose an arbitrary point \( \omega \in \Omega \), and set \( T = \omega t \).
	For brevity, let \( d = \diam(a) = \diam(b) \). By Proposition~\ref{expand},
	\[ \frac{|\partial_b(T)|}{|T|} \geq \frac{2}{2d + 1} \geq \frac{1}{2d}. \]
	As \( 1_\Omega \subseteq a \), we have \( T \subseteq Ta \).
	Now, \( T \cup \partial_b(T) = Ta \cup Ta^* \), and since \( T \) and \( \partial_b(T) \) are disjoint, we have
	\( |T| + |\partial_b(T)| = |Ta \cup Ta^*| \). As \( a \cup a^* \subseteq aa^* \), we derive \( |Ta \cup Ta^*| \leq |Taa^*| \).
	Since all relations of a scheme are regular, we have \( \| taa^* \| = |\omega taa^*| \) and \( \| t \| = |\omega t| \), hence
	\[ \frac{\| taa^* \|}{\| t \|} = \frac{|\omega taa^*|}{|\omega t|} = \frac{|Taa^*|}{|T|} \geq \frac{|T| + |\partial_b(T)|}{|T|} \geq 1 + \frac{1}{2d}. \tag{\(\star\)} \]

	Let \( k \) be the smallest positive integer such that \( \| a^{k+1} \| \leq (1 + \frac{1}{2d})^{1/2} \cdot \| a^k \| \).
	Notice that \( k = O(d \log n) \), indeed, if \( \| a^{m+1} \| > (1 + \frac{1}{2d})^{1/2} \| a^m \| \) for all \( m < k \),
	then \( \| a^{k} \| > (1 + \frac{1}{2d})^{(k-1)/2} \). Hence \( n > (1 + \frac{1}{2d})^{(k-1)/2} \) and \( k < 1 + 4d \log n \). 

	We now apply Ruzsa's triangle inequality for schemes to obtain
	\[ \| taa^* \| \cdot \| a^k \| \leq \| taa^k \| \cdot \| (a^k)^* a^* \|, \]
	and after dividing by \( \| t \| \cdot \| a^k \| \) we get:
	\[ \frac{\| taa^* \|}{\| t \|} \leq \frac{\| ta^{k+1} \|}{\| t \|} \cdot \frac{\|a^{k+1} \|}{\| a^k \| } \leq
	   \frac{\| ta^{k+1} \|}{\| t \|} \cdot \left( 1 + \frac{1}{2d} \right)^{1/2}. \]
	This inequality together with~(\(\star\)), implies that for any \( t \in S^\cup \) with \( \| t \| \leq n/2 \) we have
	\[ \frac{\|ta^{k+1} \|}{\| t \| } \geq \left( 1 + \frac{1}{2d} \right)^{1/2}. \]
	We repeatedly apply the inequality above for \( t \) equal to \( a^{k+1} \), \( a^{2(k+1)} \) etc., so if
	\( \| a^{l(k+1)} \| \leq n/2 \) for some \( l \), then
	\[ \| a^{(l+1)(k+1)} \| \geq \left( 1 + \frac{1}{2d} \right)^{(l+1)/2}. \]
	Therefore for some \( m = O(d \log n) \) we have \( \| a^{m(k+1)} \| > n/2 \). As we noted earlier, this implies \( a^{2m(k+1)} = \Omega \times \Omega \).
	Therefore \( \diamd(a) \leq 2m(k+1) = O(d^2 (\log n)^2) \) proving the claim.
\end{proof}

If \( \Gamma \) is a connected vertex-transitive graph on \( \Omega \), then as mentioned in Section~\ref{pre},
its full automorphism group induces a scheme on \( \Omega \) such that \( \Gamma \) is a relation of that scheme.
This immediately implies that the proposition above applies to \( \Gamma \) and hence Theorem~\ref{main} follows.

It is interesting to know whether one can lower the exponent of \( \diam(\Gamma) \) in the upper bound. To this end, Babai states the following problem:
\medskip

\noindent
\cite[Conjecture~6.3]{euler} \emph{Given a connected vertex-transitive graph \( \Gamma \) on \( n \) vertices, do we have \( \diamd(\Gamma) = O(\diam(\Gamma) (\log n)^C) \) for some
universal constant \( C \)?}
\medskip

As far as the author is aware, this is not known even for Cayley graphs. Note that if one drops vertex-transitivity,
then the directed diameter can grow quadratically in terms of the undirected diameter: for every \( d \geq 1 \)
Chv\'atal and Thomassen~\cite[Theorem~4]{chTh} constructed an undirected graph of diameter \( d \) such that every orientation of the graph
has directed diameter at least \( d^2/2 + d \).

\section{Negative results}\label{neg}

It is interesting whether one really needs the logarithmic factor in the upper bound on the directed diameter.
Indeed, Babai asks
\medskip

\noindent
\cite[Problem~6.4]{euler} \emph{Is it possible to bound the directed diameter of a vertex-transitive graph
in terms of its undirected diameter only?}
\medskip

We give a negative answer to that question. The crux of the argument depends on the following construction of Haight and Ruzsa~\cite{haight, ruzsa}.
Given a subset \( A \) of some abelian group written additively, let \( k \cdot A \) denote the \( k \)-fold sumset \( A + \dots + A \). 

\begin{lemma}[\cite{ruzsa}]\label{hruz}
	For any fixed \( k \geq 1 \) there exists \( \alpha_k < 1 \) such that
	for all sufficiently large integer \( q \) there is a subset \( A \subseteq \mathbb{Z}_q \)
	with \( A - A = \mathbb{Z}_q \) and \( |k \cdot A| < q^{\alpha_k} \).
\end{lemma}
\begin{corollary}
	For any \( k \) there exists a Cayley graph over an abelian group with undirected diameter
	at most \( 2 \) and directed diameter at least \( k \).
\end{corollary}

A weaker version of the previous question was also suggested in~\cite{euler}:
\medskip

\noindent
\cite[Problem~6.5]{euler} \emph{Does there exist a bound on the length of the shortest directed cycle in a vertex-transitive digraph,
depending only on the undirected diameter?}%
\medskip

The answer to this question is also negative, but we need a more involved construction.

Note that in the setting of Lemma~\ref{hruz} we can always replace \( A \) by a shift \( x + A = \{ x + a \mid a \in A \} \) where \( x \in \mathbb{Z}_q \).
Indeed, \( (x+A) - (x+A) = A - A = \mathbb{Z}_q \) and \( k \cdot (x+A) = k\cdot x + k \cdot A \),
so \( |k \cdot (x+A)| = |k \cdot A| < q^{\alpha_k} \). In particular, we can always assume that \( 0 \in A \).

\begin{proposition}\label{girthEx}
	For any \( k \geq 1 \) there exists a Cayley graph over an abelian group with undirected diameter at most~\( 2 \)
	such that the length of the shortest directed cycle is at least \( k \).
\end{proposition}
\begin{proof}
	By the previous lemma, for all sufficiently large \( q \)
	there exists a subset \( A \subseteq \mathbb{Z}_q \) such that \( 0 \in A \), \( A - A = \mathbb{Z}_q \)
	and \( |k \cdot A| < q^{\alpha_k} \) where \( \alpha_k < 1 \).
	For \( x \neq 0 \) let \( I_x = \{ x, 2 \cdot x, \dots, k \cdot x \} \) be an arithmetic progression
	in \( \mathbb{Z}_q \). We claim that for \( q \) large enough it is possible to find \( x \neq 0 \) such that
	\( I_x \subseteq \mathbb{Z}_q \setminus (k \cdot A) \).

	Indeed, set \( P_i = \{ x \in \mathbb{Z}_q \mid i \cdot x \in k \cdot A \} \) for \( i = 1, \dots, k \).
	We can assume that \( q \) is a large prime number, in particular, \( \mathbb{Z}_q \) is a field with \( q > k \).
	Therefore \( |P_i| = |k \cdot A| \) and \( |P_1 \cup \dots \cup P_k| \leq k|k \cdot A| < k q^{\alpha_k} < q-1 \)
	for \( q \) large enough. Hence there exists \( x \neq 0 \) such that \( x \notin P_1 \cup \dots \cup P_k \)
	which implies \( I_x \subseteq \mathbb{Z}_q \setminus (k \cdot A) \).

	Notice that we can assume that \( \{1, \dots, k\} = I_1 \subseteq \mathbb{Z}_q \setminus (k \cdot A) \).
	Indeed, consider \( A' = \{ x^{-1}a \mid a \in A \} \). Clearly \( A' - A' = \mathbb{Z}_q \)
	and \( k \cdot A' = \{ x^{-1}a \mid a \in k \cdot A \} \), so we may replace \( A \) by \( A' \).

	We first construct a Cayley graph with undirected diameter at most \( 4 \) and the length of the shortest nontrivial directed cycle at least~\( k \).
	Let \( V = \mathbb{Z}_q \times \mathbb{Z}_q \) be the underlying abelian group
	and let \( B = \{ (a, -1),\, (-1, a) \mid a \in A \} \) be the connection set. As \( A - A = \mathbb{Z}_q \),
	the difference \( B - B \) contains \( (\mathbb{Z}_q, 0) = \{ (x, 0) \mid x \in \mathbb{Z}_q \} \) and similarly
	\( (0, \mathbb{Z}_q) \). Now \( V = (\mathbb{Z}_q, 0) + (0, \mathbb{Z}_q) \), so \( B - B + B - B = V \)
	and the undirected diameter of the Cayley graph \( \mathrm{Cay}(V, B) \) is at most~\( 4 \).

	Consider a nontrivial directed cycle in our Cayley graph. It involves \( n \) generators
	of the form \( (a_1, -1), \dots, (a_n, -1) \) and \( m \) generators of the form \( (-1, a'_1), \dots, (-1, a'_m) \),
	where \( a_1, \dots, a_n, a'_1, \dots, a'_m \in A \). Since it is a cycle, we have the following equalities modulo \( q \):
	\[ 
	\begin{cases}
		a_1 + \dots + a_n - m = 0,\\
		-n + a'_1 + \dots + a'_m = 0.
	\end{cases}
	\]
	Assume that \( n \leq k \) and \( m \leq k \). If \( m = 0 \), then second equality implies that \( n = 0 \) modulo \( q \),
	which is a contradiction with \( q > k \) and the fact that we consider a nontrivial cycle.
	Hence \( m \geq 1 \) and, similarly, \( n \geq 1 \). Now, \( m \in n \cdot A \) and \( 0 \in A \) implies
	\( n \cdot A \subseteq k \cdot A \), so \( m \) lies in \( k \cdot A \). Recall that the interval \( \{ 1, \dots, k \} \) lies fully outside of
	\( k \cdot A \) hence we have a contradiction. Therefore \( n \) or \( m \) is larger than \( k \) and the length of any nontrivial directed
	cycle is at least~\( k \).

	Finally, to construct the required Cayley graph of undirected diameter at most~\( 2 \), use
	the argument above with \( 2k \) in place of \( k \) to construct a Cayley graph \( \mathrm{Cay}(V, B) \), \( B \subseteq V \),
	such that \( B - B + B - B = V \) and the length of the shortest nontrivial directed cycle in the graph is at least~\( 2k \).
	Now in \( \mathrm{Cay}(V, 2 \cdot B) \) a shortest nontrivial directed cycle has length at least~\( k \),
	and the undirected diameter of the graph is at most~\( 2 \), since \( 2 \cdot B - 2 \cdot B = V \). The claim is proved.
\end{proof}

The examples above are Cayley graphs over abelian groups, so it would be interesting to know if it is possible
to construct similar examples over groups which are far from being abelian, for example, over finite simple groups.

\section{Acknowledgements}

The author would like to express his gratitude to prof.\ A.V.~Vasil'ev for numerous remarks improving the text,
and to prof.\ L.~Pyber for suggesting him the statement of Proposition~\ref{expand}.

The research was carried out within the framework of the Sobolev Institute of Mathematics state contract (project FWNF-2022-0002).

\bigskip

\noindent
\emph{Saveliy V. Skresanov}

\noindent
\emph{Sobolev Institute of Mathematics, 4 Acad. Koptyug avenue, Novosibirsk, Russia}

\noindent
\emph{Email address: skresan@math.nsc.ru}

\end{document}